\documentclass{amsart}

\usepackage{amssymb}
\usepackage{amsmath}
\usepackage{amsthm}
\usepackage[all]{xy}

\newtheorem{thm}{Theorem}[section]

\newtheorem{lem}[thm]{Lemma}
\newtheorem{prop}[thm]{Proposition}
\newtheorem{cor}[thm]{Corollary}

\theoremstyle{definition}

\theoremstyle{remark}

\title{Effectiveness and strong graph indivisibility}
\author{Damir Dzhafarov, Reed Solomon, and Andrea Volpi}
\thanks{Dzhafarov and Solomon were partially supported by a Focused Research Group grant from the National Science Foundation of the 
United States, DMS-1854355. Volpi was partially funded by PRIN2022 ``Models, sets and classifications'' 2022TECZIA CUP:G53D23001890006 - M4 C2 I1.1.}

\begin{document}

\begin{abstract}
	A relational structure is \emph{strongly indivisible} if for every partition $M = X_0 \sqcup X_1$, the induced substructure on $X_0$ or $X_1$ is isomorphic to $\mathcal{M}$. Cameron (1997) showed that a graph is strongly indivisible if and only if it is the complete graph, the completely disconnected graph, or the random graph. We analyze the strength of Cameron's theorem using tools from computability theory and reverse mathematics. We show that Cameron's theorem is is effective up to computable presentation, and give a partial result towards showing that the full theorem holds in the $\omega$-model $\mathsf{REC}$. We also establish that Cameron's original proof makes essential use of the stronger induction scheme $\mathsf{I}\Sigma^0_2$.
\end{abstract}

\maketitle

\section{Introduction}
\label{sec:introduction}

Versions of Ramsey's theorem have been a consistent source of important principles in computable combinatorics and reverse mathematics. In recent years, 
Weihrauch reducibility and its variants have given new tools and perspectives to study uniformity questions in these contexts and to make fine distinction between proof methods. 
The majority of this work considers subsets of size $n \geq 2$ with an emphasis on $n=2$. However, even the $n=1$ cases, or pigeonhole principles, often have subtle 
connections to uniformity and induction. 

The canonical case is Ramsey's theorem for singletons: every finite coloring $c: \omega \rightarrow k$ has an infinite monochromatic set. 
On the reverse math side, for fixed $k$, $\mathsf{RT}^1_k$ is provable in $\mathsf{RCA}_0$, while Hirst \cite{hirst} showed the full result $\mathsf{RT}^1$ 
is equivalent to the induction scheme $\mathsf{B}\Sigma^0_2$. On the Weihrauch side, non-reductions were found for a variety of reducibilities by 
Dorais, Dzhafarov, Hirst, Mileti and Shafer \cite{DDHM}, Brattka and Rakotoniaina \cite{BR}, Hirschfeldt and Jockusch \cite{HJ}, Patey \cite{patey}, and 
Dzhafarov, Patey, Solomon and Westrick \cite{DPSW}. 

Adding some structure, Chubb, Hirst and McNicholl \cite{CHM} introduced a tree version of $\mathsf{RT}^1$: every coloring $c: 2^{<\omega} \rightarrow k$ has 
a monochromatic subset $H$ isomorphic to $2^{< \omega}$ as a partial order. In reverse math, this principle is denoted $\mathsf{TT}^1$ and was shown to 
lie strictly between the induction schemes $\mathsf{B}\Sigma^0_2$ and $\mathsf{I}\Sigma^0_2$ 
by Corduan, Groszek and Mileti \cite{CGM} and Chong, Li, Wei and Yang \cite{CLWY}. Interestingly, Ko{\l}odziejczyck has a proof (reproduced in 
\cite{DSV}) using results from \cite{CLWY} and \cite{CGM} to show that $\mathsf{TT}^1$, unlike $\mathsf{RT}^1$, is not equivalent to any arithmetical statement. 
Dzhafarov, Solomon and Valenti \cite{DSV} give a detailed account of the relationship between the singleton versions of Ramsey's theorem and the tree theorem in the 
Weihrauch degrees including a Weihrauch version of Ko{\l}odziejczyck's result. 

More generally, a relational structure $\mathcal{M}$ is \textit{indivisible} if for every finite coloring $c: M \rightarrow k$ of the domain of $\mathcal{M}$, there is a 
monochromatic subset $H$ of $M$ such that the induced substructure on $H$ is isomorphic to $\mathcal{M}$. In this setting, 
Gill \cite{Gill} considers a variety of  Weihrauch problems related to indivisible structures such as the dense linear order $(\mathbb{Q},\leq)$ and the random graph $\mathcal{R}$. 

There is a significant difference between, on one hand, Ramsey's theorem for singletons and the indivisibility of $\mathcal{R}$, and on the other hand, 
$\mathsf{TT}^1$ and the indivisibility of $(\mathbb{Q},\leq)$. For a coloring $c: \omega \rightarrow k$, there is a color $i$ such that $c^{-1}(i)$ is an infinite 
monochromatic set. 
That is, one of the colors is the desired subset for $\mathsf{RT}^1$. Similarly, for $c: \mathcal{R} \rightarrow k$, one of the colors is an isomorphic copy of $\mathcal{R}$, a fact 
first pointed out by Henson \cite{Henson}. 
However, there are colorings of $2^{<\omega}$ and $\mathbb{Q}$ such that no full color is isomorphic to $2^{<\omega}$ or $\mathbb{Q}$ as a partial or linear order 
respectively. For example, the $2$-coloring of $2^{< \omega}$ defined by $c(\sigma) = 0$ if and only if $|\sigma| \geq 1$, and the $2$-coloring of $\mathbb{Q}$ given by 
$c(q) = 0$ if and only if $q < 0$ or $q = 1$. 

Our concern here is with the stronger property that requires the full set of some color to yield an isomorphic substructure. 
In the combinatorics literature, colorings are often replaced with partitions, and a relational structure $\mathcal{M}$ is said to have the 
\textit{pigeonhole property} if for any finite partition 
$M = X_0 \sqcup \cdots \sqcup X_{k-1}$, there is an $i$ such that the induced substructure on $X_i$ is isomorphic to $\mathcal{M}$. Setting aside induction issues, 
this definition is often equivalently stated with partitions into two pieces, i.e.~$M = X_0 \sqcup X_1$. 

Unfortunately, the term pigeonhole property is frequently used in computability theory and reverse math for a Ramsey style property that only requires a 
subset of a color to 
induce an isomorphic structure. Therefore, to avoid terminological conflict, we say $\mathcal{M}$ is \textit{strongly indivisible} if for every partition $M = X_0 \sqcup X_1$, 
the induced substructure on $X_0$ or $X_1$ is isomorphic to $\mathcal{M}$. 

Strong indivisibility (under the name pigeonhole property) appears to have been introduced in Cameron \cite{Cam:97} where he proved there are exactly three 
strongly indivisible countable graphs: the complete graph $K_{\omega}$, the completely disconnected graph $\overline{K}_{\omega}$ and the random graph $\mathcal{R}$. 
There are similar classifications of the strongly indivisible tournaments, posets and linear orders in Bonato, Cameron and Deli\'{c} \cite{BCD00} as well as a  
study of the connection between Fra\"{i}sse limits and strong divisibility by Bonato and Deli\'{c} \cite{BD99}. 

Our goal is to examine Cameron's classification of the strongly indivisible countable graphs from the point of view of 
reverse mathematics and computable combinatorics. In Section \ref{sec:classical}, we give a classical proof of the classification, 
showing it can be done in $\textsf{ACA}_0$ and pointing out where it uses arithmetic comprehension and 
induction axioms beyond $\mathsf{I}\Sigma^0_1$. 

In Section \ref{sec:presentation}, we show the classification is effective up to computable presentation. That is, if $G$ is a computable graph not isomorphic to 
$K_{\omega}$, $\overline{K}_{\omega}$ or $\mathcal{R}$, then there is a computable copy $H$ of $G$ and a computable partition $H = X_0 \sqcup X_1$ such that 
neither the induced graph on $X_0$ nor the induced graph on $X_1$ is even classically isomorphic to $G$. We show the move from $G$ to $H$ is necessary to get 
this strong result 
by constructing a computable graph $G$ that is not isomorphic to $K_{\omega}$, $\overline{K}_{\omega}$ or $\mathcal{R}$, but for every computable partition 
$G = X_0 \sqcup X_1$, the induced subgraph on at least one of $X_0$ or $X_1$ is classically isomorphic to $G$. 

In Section \ref{sec:REC}, we provide a partial result towards showing the classification theorem holds in the $\omega$-model $\mathsf{REC}$. If $G$ is a computable 
graph that is not isomorphic to $K_{\omega}$, $\overline{K}_{\omega}$ or $\mathcal{R}$ and for which the set of vertices of finite degree is c.e., then we show 
there is a computable partition $G = X_0 \sqcup X_1$ such that neither $X_0$ nor $X_1$ is computably 
isomorphic to $G$. However, the full question of whether $\mathsf{REC}$ satisfies the classification theorem remains open. 

Finally, in Section \ref{sec:induction}, we return to Cameron's original proof of the classification. In second order arithmetic, the least number principle 
$\mathsf{L}\Sigma^0_2$ is the axiom scheme that asserts for each $\Sigma^0_2$ formula $\varphi(x)$, if there exists an $x$ such that $\varphi(x)$, then 
there is a least such $x$. Over $\mathsf{RCA}_0$, $\mathsf{L}\Sigma^0_2$ is equivalent to $\mathsf{I}\Sigma^0_2$, the usual induction scheme for 
$\Sigma^0_2$ formulas, and hence is not provable in $\mathsf{RCA}_0$. The classical proof of the classification theorem applies $\mathsf{L}\Sigma^0_2$ to a 
graph $G \not \cong \mathcal{R}$ to obtain the smallest size counterexample in $G$ to the extension property that characterizes the random graph. In 
Section \ref{sec:induction}, we show that the existence of a counterexample to the extension property of minimal size in every non-random graph is in fact equivalent to 
the full induction scheme $\mathsf{L}\Sigma^0_2$. 

Our notation follows the standard references such as Dzhafarov and Mummert \cite{DM22}, Hirschfeldt \cite{hirschfeldt}, Simpson \cite{sim:book}, and 
Soare \cite{soa:book}. We introduce the graph theory terminology in the next section.

\section{The classical proof}
\label{sec:classical}

In this section, we give Cameron's proof classifying the countable strongly indivisible graphs with an eye towards formalizing it in reverse mathematics. 
In the context of a model of a subsystem of second order arithmetic, we let $\mathbb{N}$ denote the first order part of the model, with the understanding that 
$\mathbb{N}$ is standard when we  work outside a formal setting. 

Formally, a \textit{graph} $G = (V,E)$ is a pair consisting of a nonempty set $V \subseteq \mathbb{N}$ of vertices and an irreflexive, symmetric edge relation $E$. 
For $X \subseteq V$, the \textit{induced subgraph} is $(X, E \! \upharpoonright \! X \times X)$. We frequently abuse notation by equating a graph with its 
domain (e.g.~writing $X \subseteq G$), by conflating a set of vertices with its induced subgraph, and by using subgraph to mean induced subgraph.

A vertex $x$ is \textit{isolated} if $\neg E(x,y)$ for all $y \in G$ and is \textit{universal} if $E(x,y)$ for all $y \neq x$ in $G$.
A countable graph $G = (V,E)$ is \textit{strongly indivisible} if for every vertex partition $V = X_0 \sqcup X_1$, either $X_0 \cong G$ or $X_1 \cong G$. 

For $n > 0$, $K_n$ denotes the complete graph on $n$ vertices. 
$K_{\omega}$ denotes the complete graph with $V = \mathbb{N}$ and $E = \{ \langle n,m \rangle : m \neq n \}$, and $\overline{K}_{\omega}$ denotes 
the completely disconnected graph with $V = \mathbb{N}$ and $E = \emptyset$. $\mathsf{RCA}_0$ proves that each of these graphs is  
strongly indivisible. 

\begin{prop}[$\mathsf{RCA}_0$]
$K_\omega$ and $\overline{K}_\omega$ are strongly indivisible. 
\end{prop}

\begin{proof}
Fix a partition of the vertices $\mathbb{N} = X_0 \sqcup X_1$. At least one of $X_0$ and $X_1$ is infinite, so the corresponding subgraph is 
isomorphic to $K_{\omega}$ or to $\overline{K}_{\omega}$.  
\end{proof}

This symmetry between $K_{\omega}$ and $\overline{K}_{\omega}$ with respect to strong indivisibility extends more generally. 
For a graph $G = (V,E)$, let $\overline{G}$ denote the graph 
obtained by swapping the edges and non-edges, except along the diagonal. 
Formally, $\overline{G} = (V, \overline{E})$ with $\overline{E} = \{ \langle m,n \rangle \not \in E : m \neq n \}$.

\begin{prop}[$\mathsf{RCA}_0$]
\label{prop:duality}
$G$ is strongly indivisible if and only if $\overline{G}$ is strongly indivisible. 
\end{prop}

\begin{proof}
Let $V = X_0 \sqcup X_1$ be a partition of the vertices, and let $H_i$ and $\overline{H}_i$ denote the corresponding subgraphs in $G$ and $\overline{G}$. 
Because a graph isomorphism preserves both edges and non-edges, a 
bijection $f: X_i \rightarrow V$ is an isomorphism from $H_i$ to $G$ if and only if it is an isomorphism from $\overline{H}_i$ to $\overline{G}$. 
\end{proof}

The standard development of the random graph can be carried out in $\mathsf{RCA}_0$. Let $\varphi_{\mathcal{R}}$ denote the usual extension axiom 
stated in second order arithmetic. 
\[
\forall \, \text{finite} \, A, B \subseteq V \, \Big( A \cap B = \emptyset \rightarrow \exists x \, \big[ (\forall a \in A) \, E(x,a) \wedge (\forall b \in B) \, \neg E(x,b) \big] \Big) 
\]
A countable graph satisfying $\varphi_{\mathcal{R}}$ is called a \textit{random graph}. Although this axiomatization is standard, it is typical when working in a 
random graph to assume the existential witness $x$ satisfies $x \not \in A \cup B$. To see why this condition is admissible, 
apply $\varphi_{\mathcal{R}}$ to $A \cup B$ and $\emptyset$ 
to get a vertex  $v$ connected to every node in $A \cup B$. Then apply $\varphi_{\mathcal{R}}$ to $A$ and $B \cup \{ v \}$ to get a node $x$ connected to everything in $A$ and 
nothing in $B \cup \{ v \}$. It follows that $x \not \in A$ because the edge relation is irreflexive and $x \not \in B$ because 
everything in $B$ is connected to $v$. 

There are a number of places below where it is important that we require the existential witness $x$ to be outside $A \cup B$. Therefore, going forward, we adopt the 
convention that the axiom $\varphi_{\mathcal{R}}$ includes the stipulation that $x \not \in A \cup B$. 

$\mathsf{RCA}_0$ suffices to show there is a random graph by, 
for example, letting $V = \mathbb{N}$ and putting a symmetric edge between $x$ and $y$ when $x < y$ and the $x$-th bit of the binary representation of $y$ is 1. 
Moreover, it is a folklore (and easily checked) result that the classical back-and-forth argument can be carried out in 
$\mathsf{RCA}_0$ to show that there is a unique random graph up to isomorphism. 

\begin{prop}[$\mathsf{RCA}_0$]
\label{prop:Runique}
If $G_0$ and $G_1$ are random graphs, then $G_0 \cong G_1$. 
\end{prop}

We continue to use $\mathcal{R}$ to denote a random graph, which by Proposition \ref{prop:Runique}, is determined up to isomorphism in 
$\mathsf{RCA}_0$. 

\begin{prop}[$\mathsf{RCA}_0$]
Let $\mathcal{R}$ be a random graph. 
\begin{enumerate}
\item[(1)] Let $A$ and $B$ be disjoint finite sets of vertices in $\mathcal{R}$. The subgraph $G_{A,B}$ on 
$V_{A,B} = \{ x \in \mathcal{R} \setminus (A \cup B) : (\forall a \in A) E(x,a) \wedge 
(\forall b \in B) \neg E(x,b) \}$ is a random graph. 
\item[(2)] $\mathcal{R}$ is strongly indivisible.
\end{enumerate}
\end{prop}

\begin{proof}
For (1), to show $G_{A,B}$ is a random graph, consider disjoint finite sets $C,D \subseteq V_{A,B}$. Since $\mathcal{R}$ is random, there is 
a node $x \not \in A \cup B \cup C \cup D$ such that $E(x,y)$ for all $y \in A \cup C$ and $\neg E(x,z)$ for all $z \in B \cup D$. It follows that $x \in V_{A,B}$ and that $x$ witnesses the 
extension axiom for the finite sets $C$ and $D$ in $G_{A,B}$.

For (2), fix a partition $\mathcal{R} = X_0 \sqcup X_1$. Suppose for a contradiction that neither of the induced subgraphs are random. For $i < 2$, fix disjoint 
finite sets $A_i, B_i \subseteq X_i$ for which the extension axiom fails in the induced subgraph. Since $\mathcal{R}$ is random, there is an $x \not \in A_0 \cup A_1 \cup B_0 
\cup B_1$ such that $E(x,a)$ for all $a \in A_0 \cup A_1$ and $\neg E(x,b)$ for all $b \in B_0 \cup B_1$. Let $j < 2$ be such that $x \in X_j$ and notice that $x$ witnesses the 
extension property for the pair $A_j, B_j$ in $X_j$, contrary to assumption.  
\end{proof}

We turn to showing $K_{\omega}$, $\overline{K}_{\omega}$ and $\mathcal{R}$ are the only strongly indivisible countable graphs. Since every nontrivial 
partition of a finite graph witnesses that it is not strongly indivisible, we only need to consider infinite graphs. 

\begin{thm}[Cameron \cite{Cam:97}]
\label{thm:Cameron}
If $G$ is a countable strongly indivisible graph, then $G$ is isomorphic to $K_{\omega}$, $\overline{K}_{\omega}$ or $\mathcal{R}$.
\end{thm}

\begin{proof}
Assume $G$ is an infinite graph that is not isomorphic to $K_{\omega}$, $\overline{K}_{\omega}$ or $\mathcal{R}$. 

\textit{Case 1.} Suppose $G$ has isolated vertices. Let $X_0 = \{ x \in G : x \text{ is isolated} \}$ and $X_1 = G \setminus X_0$. The induced 
subgraph on $X_0$ is $\overline{K}_{|X_0|}$, which is not isomorphic to $G$ by assumption. If $x$ were an isolated vertex in the induced subgraph $X_1$, then, 
in fact, $x$ would be an isolated vertex in $G$, and hence $x \in X_0$. Therefore, $X_1$ has no isolated vertices and so is not isomorphic to $G$.

\textit{Case 2.} Suppose $G$ has universal vertices. This case follows by letting $X_0$ be the set of universal vertices and reasoning as in Case 1. 

\textit{Case 3.} Suppose $G$ has neither isolated nor universal vertices. Since $G$ is not a random graph, let $n$ be least such that  
there are disjoint sets $A$ and $B$ for which $|A|+|B| = n$ and the extension axiom $\varphi_{\mathcal{R}}$ 
fails for $A$ and $B$.

We claim that $n \geq 2$. We cannot have $n=0$ because $G$ is nonempty, so suppose $n=1$. 
If $A = \emptyset$ and $B = \{ b \}$, then the failure of the extension axiom for $A, B$ means $b$ is a universal vertex, contrary to 
our case assumption. On the other hand, if $A = \{ a \}$ and $B = \emptyset$, then the failure of $\varphi_{\mathcal{R}}$ implies $a$ is an isolated vertex, 
also contrary to our case assumption.

Since $n \geq 2$, we can partition $A \cup B = U_0 \sqcup U_1$ into nonempty sets $U_0$ and $U_1$. We say a vertex $x$ is 
\textit{not correctly joined 
to} $U_i$ if there is an $a \in U_i \cap A$ such that $\neg E(x,a)$ or there is a $b \in B \cap U_i$ such that $E(x,b)$. Since the extension axiom does not hold 
for $A$ and $B$, every vertex $x \not \in A \cup B$ is not correctly joined to at least one of $U_0$ or $U_1$. 

Let $X_0 = U_0 \cup \{ x \in G : x \not \in U_1 \wedge x \text{ is not correctly joined to } U_0 \}$ and let $X_1 = G \setminus X_0$. Note that 
$G = X_0 \sqcup X_1$, $U_1 \subseteq X_1$, and every node in $X_1 \setminus U_1$ is not correctly joined to $U_1$. By construction,  
$X_0$ fails to satisfy the extension axiom with $A \cap U_0$ and $B \cap U_0$, while $X_1$ fails to satisfy the extension axiom 
with $A \cap U_1$ and $B \cap U_1$. Since $U_0$ and $U_1$ are nonempty, it follows that $X_0$ and $X_1$ fail to satisfy instances of the 
extension axiom for which the sum of the sizes of the witnessing sets is strictly less than $n$. Therefore, since $n$ was chosen least for $G$, 
neither $X_0$ nor $X_1$ is isomorphic to $G$.
\end{proof}

On its face, this proof uses axioms outside of $\mathsf{RCA}_0$ in three places. In Cases 1 and 2, it uses 
arithmetic comprehension to form the sets of isolated and universal vertices. In Case 3, it uses the existence of the least $n$ 
such that 
\[
\exists \, \text{finite disjoint} \, A, B \, \Big(  |A|+|B|=n \, \wedge \, \forall x \, \big[ (\exists a \in A) \, \neg E(x,a) \vee (\exists b \in B) \, 
E(x,b) \big] \Big) 
\]
Since the least number axiom schema $\mathsf{L}\Sigma^0_2$ is equivalent to $\mathsf{I}\Sigma^0_2$, which holds in $\mathsf{ACA}_0$,  
Theorem \ref{thm:Cameron} is provable in $\mathsf{ACA}_0$. 

The existence of the set of isolated or universal vertices is equivalent to $\mathsf{ACA}_0$. Similarly, $\mathsf{L}\Sigma^0_2$ is 
equivalent to the least number principle restricted to formulas of the form above. 
We give the equivalence with $\mathsf{ACA}_0$ here because it is short, but delay the $\mathsf{L}\Sigma^0_2$ 
equivalence until Section \ref{sec:induction}. These equivalences do not give us lower bounds on the strength of Theorem \ref{thm:Cameron}, but they do tell us that the 
proof above cannot be carried out in a system weaker than $\mathsf{ACA}_0$.

\begin{prop}[$\mathsf{RCA}_0$]
The following are equivalent. 
\begin{enumerate}
\item[(1)] $\mathsf{ACA}_0$.
\item[(2)] For every graph, the set of universal nodes exists. 
\item[(3)] For every graph, the set of isolated nodes exists. 
\end{enumerate}
\end{prop}

\begin{proof}
The implications from (1) to (2) and from (1) to (3) hold because these sets are arithmetically definable. 

For the implication from (2) to (3), fix a graph $G$. The set of isolated nodes in $G$ is the same as the set of universal nodes in $\overline{G}$, which exists by (2). 

For the implication from (3) to (1), let $f: \mathbb{N} \rightarrow \mathbb{N}$ be an arbitrary one-to-one function. It suffices to show the range of $f$ exists. 
Define $G$ with $V = \mathbb{N}$ and a symmetric edge between $x$ and $y$ if and only if $x=2n$, $y=2m+1$ and $f(n)=m$. The range of $f$ 
is definable from the set of isolated nodes in $G$ because an odd vertex $2m+1$ is isolated if and only if $m$ is not in the range of $f$. 
\end{proof}

\section{Effectiveness up to presentation}
\label{sec:presentation}

Our motivating question is whether the classification of strongly indivisible graphs is provable in $\mathsf{RCA}_0$, and in particular, 
whether it holds in the $\omega$-model $\mathsf{REC}$. In the previous section, we showed one direction holds in $\mathsf{RCA}_0$, which translates to 
$\mathsf{REC}$ as follows. 

\begin{thm}
\label{thm:easydirection}
Let $G$ be a computable graph that is isomorphic to $K_{\omega}$, $\overline{K}_{\omega}$ or $\mathcal{R}$. For every computable partition $G = X_0 \sqcup X_1$, 
either $X_0$ or $X_1$ is computably isomorphic to $G$. 
\end{thm}

Because the graphs $K_{\omega}$, $\overline{K}_{\omega}$ and $\mathcal{R}$ are computably categorical, it does not matter whether we use ``isomorphic'' or 
``computably isomorphic'' in the statement of Theorem \ref{thm:easydirection}. However, in general, it is possible for the effectiveness properties to vary across 
computable presentations of a graph. We show that Theorem \ref{thm:Cameron} is effective up to computable presentation in a strong 
form in which we consider the classical isomorphism types of the partition pieces.  

\begin{thm}
\label{thm:presentation}
Let $G$ be a computable graph that is not isomorphic to $K_{\omega}$, $\overline{K}_{\omega}$ or $\mathcal{R}$. There is a computable presentation $H$ of $G$ 
and a computable partition $H = X_0 \sqcup X_1$ such that neither $X_0$ nor $X_1$ is classically isomorphic to $G$. 
\end{thm}

To prove Theorem \ref{thm:presentation}, we use the following theorem and corollary to mimic the classical proof of Theorem \ref{thm:Cameron}. 

\begin{thm}
\label{thm:isolated}
Every computable graph $G$ has a computable copy in which the set of isolated vertices is computable. 
\end{thm}

\begin{cor}
\label{cor:universal}
Every computable graph $G$ has a computable copy in which the set of universal vertices is computable. 
\end{cor}

Corollary \ref{cor:universal} follows immediately from Theorem \ref{thm:isolated} by shifting from $G$ to $\overline{G}$. 
At the end of the section, we return to a proof of Theorem \ref{thm:isolated}. For now, we use these results to prove 
Theorem \ref{thm:presentation}. 

\begin{proof}
We follow the classical proof of Theorem \ref{thm:Cameron} given above. In Case 1, when $G$ has isolated vertices, we apply Theorem \ref{thm:isolated} 
to get a computable copy $H$ for which the partition $H = X_0 \sqcup X_1$ is computable, where $X_0$ is the set of isolated vertices. In Case 2, when $G$ 
has universal vertices, we apply Corollary \ref{cor:universal} to get a computable copy $H$ for which the partition $H = X_0 \sqcup X_1$ is computable, where 
$X_0$ is the set of universal vertices. In either case, the proof that neither $X_0$ nor $X_1$ is classically isomorphic to $G$ is the same as in 
Theorem \ref{thm:Cameron}. 

For Case 3, when $G$ has neither isolated nor universal vertices, we run the argument from Theorem \ref{thm:Cameron} without changing 
the presentation of $G$. The least value $n$ exists in the standard natural numbers, and the partition pieces $X_0$ and $X_1$ are 
computable because they are defined with bounded quantifiers. 
\end{proof}

Our next goal is to show the shift of computable presentations in 
Theorem \ref{thm:presentation} is necessary to get a strong effectiveness result that considers the partition pieces up to classical isomorphism. 

Let $K_{< \omega}^{\infty}$ denote the graph consisting of infinity many disjoint copies of $K_n$ for each $n \geq 1$. We say that a copy of $K_n$ inside 
$K_{< \omega}^{\infty}$ is \textit{finished} if it is not a subgraph of a larger $K_m$ inside $K_{< \omega}^{\infty}$. 

There is a nice computable copy $H$ of $K_{< \omega}^{\infty}$ for which there is a computable function $f$ such that each vertex $x$ sits in a 
finished copy of $K_{f(x)}$. There are many computable partitions $H = X_0 \sqcup X_1$ such that neither 
$X_0$ nor $X_1$ is classically 
isomorphic to $K_{< \omega}^{\infty}$. For example, let $X_0$ be the set of isolated nodes (i.e.~those for which $f(x)=1$), or more generally, let 
$X_0$ be the set of all nodes for which $f(x)=n$ for any fixed $n$. 

However, $K_{< \omega}^{\infty}$ also has computable copies which are less uniformly constructed. In the next theorem, we build a computable $G \cong 
K_{< \omega}^{\infty}$ such that 
every computable partition $G = X_0 \sqcup X_1$ has at least one $X_i$ classically isomorphic to $K_{< \omega}^{\infty}$. 

The isomorphism type of $K_{< \omega}^{\infty}$ has several properties that make it suitable for this construction. Deleting finitely many vertices 
doesn't change its isomorphism type, and neither does adding countably many disjoint copies of finished graphs $K_n$ for each $n \in \omega$. 
Moreover, a subgraph of $K_n$ is isomorphic to $K_m$ for some $m \leq n$. Therefore, if $K_{< \omega}^{\infty} = X_0 \sqcup X_1$, 
then each finished component of $X_i$ is a copy of $K_m$ for some $m$. It follows that $X_i$ is isomorphic to $K_{< \omega}^{\infty}$ 
as long as it contains infinitely many finished copies of $K_m$ for each $m$, i.e.~we do not have to worry about what other finished components $X_i$ contains. 

In the proof of the following theorem, it is convenient to regard each $\Phi_e$ as $\{ 0,1 \}$-valued and to use $K_0$ to denote the empty set. 

\begin{thm}
\label{thm:diagonalize}
There is a computable $G \cong K_{< \omega}^{\infty}$ such that for every computable partition $G = X_0 \sqcup X_1$, either $X_0$ or $X_1$ is classically 
isomorphic to $K_{< \omega}^{\infty}$.
\end{thm}

\begin{proof}
We build $G$ computably in stages with $G_s$ denoting the graph at the end of stage $s$. The set of vertices in $G_s$ will be a finite initial segment of 
$\omega$. We neither add nor delete edges between vertices in $G_s$ after stage $s$. 

Let $\pi_1$ denote the projection function onto the first coordinate. As $s$ goes to 
infinity, the values of $\pi_1(s)$ hit each natural number infinitely often. We use this property to ensure $G$ is isomorphic to $K_{< \omega}^{\infty}$. At the start of 
stage $s$, we add $\pi_1(s)+1$ new vertices and put edges between them to form a finished copy of $K_{\pi_1(s)+1}$. This action ensures $G$ 
has a subgraph isomorphic to $K_{< \omega}^{\infty}$. Therefore, as long as each additional finished component in $G$ has the form $K_n$ for some $n$, 
$G$ will be isomorphic to $K_{< \omega}^{\infty}$. 

For each index $e$, we let $X^e_0 = \{ x : \Phi_e(x) = 0 \}$ and $X^e_1 = \{ x : \Phi_e(x) = 1 \}$. If $\Phi_e$ is total, these sets partition $G$. We list our  
requirements as follows. 
\[
R_e: \text{If } \Phi_e \text{ is total}, \text{ then } X^e_0 \cong K_{< \omega}^{\infty} \text{ or } X^e_1 \cong K_{< \omega}^{\infty}.
\]
To satisfy this requirement, it suffices to ensure that at least one $X_i$ contains infinitely many finished copies of $K_n$ for each $n$. 

The $R_e$ module keeps three parameters: numbers $m^e_0$ and $m^e_1$, and a finite set $C^e$. Each parameter will change during the construction. We 
typically suppress denoting the stage, but write $m^e_{i,s}$ and $C^e_s$ (and later $C^e_{i,s}$) when we need to explicitly reference the stage. 
The numbers $m^e_0$ and $m^e_1$ track the finished graphs $K_n$ we 
have seen in $X_0$ and $X_1$ respectively. For each $n < m^e_i$, we will already have forced a finished copy of $K_{\pi_1(n)+1}$ into $X_i$ with 
separate copies when $n \neq n'$ and $\pi_1(n) = \pi_1(n')$. The current goal for $R_e$ is to create a finished copy of $K_{\pi_1(m^e_0)+1}$ in $X_0$ or 
$K_{\pi_1(m^e_1)+1}$ in $X_1$.  

To meet this goal, add a new element $y_0$ to $G_s$ with no edges, set $C^e = \{ y_0 \}$ and $C^e_i = \{ y \in C^e : \Phi_e(y) = i \}$ for $i < 2$. 
Currently, $C^e$ is a copy of $K_1$ and each $C^e_i$ is a copy of $K_0$ (i.e.~is empty). If $\Phi_e(y_0)$ halts at a future stage $s_0$,  
one of the $C^e_i$ sets becomes a copy of $K_1$ and the other remains a copy of $K_0$. If $C^e_i \cong K_{\pi_1(m^e_i)+1}$ for an $i < 2$, 
then we have met our goal on the $X_i$ side. In this case, set $m^e_i = m^e_i+1$, empty $C^e$, leave 
$m^e_{1-i}$ unchanged, and restart the $R_e$ module with the new parameters. 

Otherwise, we add a new element $y_1$ to $G_{s_0}$, connect it to $y_0$, and expand $C^e = \{ y_0, y_1 \}$ to a copy of $K_2$. If   
$\Phi_e(y_1)$ later halts at $s_1$, one of the $C^e_i$ sets grows by one element. If $C^e_i \cong K_{\pi_1(n_i)+1}$, then we have met our goal 
on the $X_i$ side, and so we increment $m^e_i$, empty $C^e$, and restart the $R_e$ module with the new parameters. If we have not met the goal 
on either side, add a new vertex $y_2$ to $G_{s_1}$, connect it to $y_0$ and $y_1$, expand $C^e = \{ y_0, y_1, y_2 \}$ to a copy of $K_3$, and repeat the 
process above.

We cannot cycle through this process infinitely often because when $\Phi_e(y_k)$ halts, we have $|C^e_0|+|C^e_1|=k+1$. Therefore, before 
$|C^e| = \pi_1(m^e_0) + \pi_1(m^e_1) +2$, one of the $C^e_i$ sets must reach $|C^e_i| = \pi_1(m^e_i)+1$, and so satisfies 
$C^e_i \cong K_{\pi_1(m^e_i)+1}$, meeting our goal on the $X_i$ side. 

When we restart the $R_e$ module at a stage $s$, we empty $C^e$ (i.e.~set $C^e_s = \emptyset$) and begin again with $C^e_s$ starting a new connected 
component. After this stage, we never add vertices to the old component $C^e_{s-1}$. Therefore, $C^e_{s-1}$ is a finished component in $G$, each  
$C^e_{i,s-1}$ is finished in $X_i$, and so we have created a finished copy of $K_{\pi_1(m^e_{i,s-1})+1}$ in $X_i$ for the $i < 2$ such that 
$C^e_{i,s-1} \cong K_{\pi_1(m^e_{i,s-1})+1}$. 

Furthermore, each time we restart the $R_e$ module, one of the $m^e_i$ parameters is incremented. 
Therefore, if $\Phi_e$ is total, the values of at least one $m^e_i$ go to infinity, causing $\pi_1(m^e_i)$ to cycle through each number 
infinitely often. It follows that $X_i$ contains infinitely many finished copies of $K_n$ for each $n \geq 1$ and hence is isomorphic to $K_{<\omega}^{\infty}$. 

The formal construction proceeds as follows. At stage 0, set $G_0 = C^0 = \{ 0 \}$, set $m^e_0 = m^e_1 = 0$ for all $e$, and set $C^e = \emptyset$ for $e > 0$. 

At stage $s > 0$, let $k = \pi_1(s)+2+|\{ e<s : \Phi_{e,s}(\max C^e) \text{ halts} \}|$, $X_s = \{ x_0, \cdots, x_{k-1} \}$ be the $k$ least unused numbers, and 
$G_s = G_{s-1} \cup X_s$. Add edges between $x_i$ and $x_j$ for $i \neq j \leq \pi_1(s)$ to create a finished copy of $K_{\pi_1(s)+1}$. 
Set $C^s = \{ x_{\pi_1(s)+1} \}$ and leave the remaining parameters for $e \geq s$ unchanged. 

Consider the indices $e < s$ in order. If $\Phi_{e,s}(\max C^e)$ does not halt, leave $m^e_i$ and $C^e$ unchanged and go to $e+1$.  
If $\Phi_{e,s}(\max C^e)$ halts, then check whether $C^e_i = \{ z \in C^e : \Phi_{e,s} = i \} \cong K_{\pi_1(m^e_i)+1}$ for some $i < 2$. If 
not, set $C^e_s = C^e_{s-1} \cup \{ x_\ell \}$ where $x_{\ell}$ is the least unused number from $X_s$. Connect $x_{\ell}$ to each element of $C^e_{s-1}$ so that 
$C^e_s \cong K_{|C^e|}$. Leave $m^e_0$ and $m^e_1$ unchanged and go to $e+1$. Finally, if $C^e_i \cong K_{\pi_1(m^e_i)+1}$, then 
set $C^e_s = \{ x_{\ell} \}$, $m^e_{i,s} = m^e_{i,s-1}+1$, $m^e_{1-i,s} = m^e_{1-i,s-1}$, and go to $e+1$. 

This completes the formal construction. The details of the verification are essentially contained in the informal description 
above as there is no interaction between the requirements with different indices. 
\end{proof}

Recall that Theorem \ref{thm:isolated} says every computable graph has a computable copy in which the set of isolated nodes is 
computable. We end this section with its proof.

\begin{proof}
Fix a computable graph $G$ and assume without loss of generality that the set of vertices is $\omega$. Suppose the set of isolated nodes is not computable, and 
hence there are infinitely many isolated nodes as well as infinitely many non-isolated nodes. Let $G_s$ denote the subgraph on $\{ 0, \ldots, s \}$. 

We build a computable graph $H$ and a $\Delta^0_2$ isomorphism $f: G \rightarrow H$ in stages such that the isolated nodes in $H$ are exactly the even numbers. 
At stage $s$, we define an injection $f_s$ on $G_s$ and let $H_s$ denote the range of $f_s$ with edge relation defined by $E_{H_s}(n,m)$ 
if and only if $E_{G_s}(f_s^{-1}(n),f_s^{-1}(m))$. Thus, by definition, $f_s$ will be an isomorphism from $G_s$ to $H_s$. 
To make the edge relation on $H$ computable, we ensure that $E_{H_s}(n,m)$ holds if and only if $E_{H_t}(n,m)$ holds for all $t \geq s$ such that $n,m \in H_t$. 

The domains of the graphs $H_s$ will not necessarily be monotonic. Suppose $x$ is isolated in $G_s$, so we set $f_s(x) = 2n$ to map $x$ to 
an even number in $H_s$. If we discover $x$ is not isolated in $G_{s+1}$ by seeing $E_G(x,s+1)$, then we need to shift $f_{s+1}(x)$ to an odd number.  
We collect the nodes $x_0 < \cdots < x_\ell$ that were isolated in $G_s$ but are connected 
to $s+1$, and we map these nodes to the least odd numbers not in $H_s$. Next, we collect the nodes $a_0 < \cdots < a_j$ (if any) that remain isolated in 
$G_{s+1}$ and map these elements onto an initial segment of the even numbers. Since the number of isolated nodes has gone down from $G_s$ to $G_{s+1}$, 
at least one even number in $H_s$ is no longer in $H_{s+1}$. However, because $G$ has infinitely many isolated nodes, each even number will eventually be 
permanently in the range of the $f_s$ maps. 

We now give the construction. At stage $0$, set $f_0(0) = 0$, noting that $0$ is isolated in $G_0$. 
At stage $s+1$, we define $f_{s+1}$ as follows.

Case 1: $s+1$ is isolated in $G_{s+1}$. Let $m$ be the least even number such that $m \not \in H_s$. Define $f_{s+1}(x) = f_s(x)$ for 
$x \leq s$ and $f_{s+1}(s+1) = m$. 

Case 2: $s+1$ is not isolated in $G_{s+1}$ but is not attached to any nodes which are isolated in $G_s$. Let $k$ be the least odd number such that $k \not \in H_s$. 
Define $f_{s+1}(x) = f_s(x)$ for $x \leq s$ and $f_{s+1}(s+1) = k$. 

Case 3: $s+1$ is not isolated in $G_{s+1}$ and it is attached to at least one node which is isolated in $G_s$. Let $x_0 < \cdots < x_\ell$ denote the nodes which are 
isolated in $G_s$ but are connected to $s+1$ in $G_{s+1}$. Let $a_0 < \ldots < a_j$ denote the nodes (if any) which are isolated in $G_{s+1}$. 
Let $k_0 < \cdots < k_{\ell+1}$ denote the least odd numbers not in $H_s$. For $x \leq s+1$, define 
\[
f_{s+1}(x) = \left\{ 
\begin{array}{ll}
2i & \text{if } x = a_i \\
k_i & \text{if } x = x_i \\
k_{\ell+1} & \text{if } x = s+1 \\
f_s(x) & \text{otherwise}
\end{array}
\right.
\]

This completes the construction. It is straightforward to check a number of properties by induction on $s$. 
First, each function $f_s$ is injective. Second, $H_s$ consists of the union of an initial segment of the even numbers 
and an initial segment of the odd numbers. Third, $x$ is isolated in $G_s$ if and only if $f_s(x)$ is even. 
Therefore, if $n \in H_s$ is even, then $\neg E_{H_s}(n,m)$ for all $m \in H_s$. 
Fourth, if $m \in H_s$ and $m \not \in H_{s+1}$, then $m$ is even. Fifth, if $m$ is odd and $f_s(x)=m$, then $f_t(x)=m$ for all $t \geq s$.

\begin{lem}
\label{lem:intersect}
For $s < t$ and $m,n \in H_s \cap H_t$, $E_{H_s}(m,n)$ if and only if $E_{H_t}(m,n)$.
\end{lem}

\begin{proof}
If $m$ or $n$ is even, then by the third property above, $\neg E_{H_s}(n,m)$ and $\neg E_{H_t}(n,m)$. Therefore, assume $m$ and $n$ are odd. 
Fix $x,y \in G_s$ with $f_s(x) = m$ and $f_s(y) = n$. By the fifth property, $f_t(x) = m$ and $f_t(y) = n$, and so by definition, 
$E_{H_s}(n,m)$ and $E_{H_t}(n,m)$ are each equivalent to $E_G(x,y)$. 
\end{proof}

\begin{lem}
\label{lem:limits}
For each $x$, there is a stage $s$ such that $f_t(x) = f_s(x)$ for all $t \geq s$.
\end{lem}

\begin{proof}
Suppose $x$ is not isolated and let $y$ be the least node such that $E_G(x,y)$. If $y < x$, then 
$x$ is not isolated in $G_x$ and therefore $f_x(x)$ is odd. If $x < y$, then at stage $y$, the construction acts in Case 3 and the value $f_y(x)$ is odd. 
In either case, once $x$ is mapped to an odd number, $f_s(x)$ has stabilized. 

Suppose $x$ is isolated and so $f_x(x)$ is even. For $s \geq x$, $f_{s+1}(x) \neq f_s(x)$ only if a node $y < x$ is isolated 
in $G_s$ but is connected to $s+1$. In this case, $f_{s+1}(y)$ becomes odd and $f_{s+1}(x) < f_s(x)$ is a smaller even number. This drop can happen at most
 finitely often before reaching a limiting value. 
\end{proof}

\begin{lem}
\label{lem:onto}
For each $n$, there is an $x \in G$ and a stage $s$ such that $f_t(x) = n$ for all $t \geq s$.
\end{lem}

\begin{proof}
Suppose $n$ is odd. Fix $s$ such that $G_s$ contains at least $(n+1)/2$ non-isolated nodes.  
Since $f_s$ maps the non-isolated nodes of $G_s$ onto an initial segment of the 
odd numbers, there is an $x \in G_s$ such that $f_s(x) = n$. Because $n$ is odd, $f_t(x) = f_s(x)=n$ for all $t \geq s$.

Suppose $n$ is even. Let $a_0 < \cdots < a_{n/2}$ be an initial segment of the isolated nodes in $G$. Fix $s$ such that these nodes form an initial segment of the isolated 
nodes in $G_s$. For every $t \geq s$, $f_t$ maps these nodes onto an initial segment of the even numbers, and therefore, $f_t(a_{n/2}) = n$ for all $t \geq s$. 
\end{proof}

Define $H = (\omega, E_H)$ with $E_H(n,m)$ holds if and only if $E_{H_s}(n,m)$ holds for the least $s$ with $n,m \in H_s$. 
By Lemma \ref{lem:intersect}, $E_H(n,m)$ holds if and only if $E_{H_s}(n,m)$ holds for some, or equivalently all,  
$s$ with $n,m \in H_s$. It follows that $n$ is isolated in $H$ if and only if $n$ is even. 

By Lemmas \ref{lem:limits} and \ref{lem:onto}, the function $f = \lim_s f_s$ is total and onto $\omega$. It is injective because each $f_s$ is injective, so 
$f: \omega \rightarrow \omega$ is a bijection. To finish the proof, we show that $f$ is an isomorphism between $G$ and $H$. 

\begin{lem}
$f:G \rightarrow H$ is an isomorphism. 
\end{lem}

\begin{proof}
Fix $x,y \in G$ and $s \geq \max \{ x,y \}$ such that $f(x) = f_s(x)$ and $f(y) = f_s(y)$.
\[
E_G(x,y) \Leftrightarrow E_{G_s}(x,y) \Leftrightarrow E_{H_s}(f_s(x),f_s(y)) \Leftrightarrow E_H(f(x),f(y)).
\]
The first equivalence follows from $x,y \in G_s$, the second follows from the definition of $E_{H_s}$, and the third follows because $f(x) = f_s(x)$ and 
$f(y) = f_s(y)$.  
\end{proof}

This completes the proof of Theorem \ref{thm:isolated}. 
\end{proof}

\section{Towards an analysis in $\mathsf{REC}$}
\label{sec:REC}

Theorem \ref{thm:Cameron} holds in $\mathsf{REC}$ if and only if for every computable graph $G$ not isomorphic to $K_{\omega}$, 
$\overline{K}_{\omega}$ or $\mathcal{R}$, there is a computable partition $G = X_0 \sqcup X_1$ such that neither $X_0$ nor $X_1$ is computably isomorphic to $G$. 
(Recall that we do not need to say ``$G$ is not computably isomorphic to $K_{\omega}$, $\overline{K}_{\omega}$ and $\mathcal{R}$'' because 
$K_{\omega}$, $\overline{K}_{\omega}$ 
and $\mathcal{R}$ are computably categorical.) While this full statement remains open, we handle a special case in this section. 

If $G$ has no isolated or universal vertices, then as noted in the proof of Theorem \ref{thm:presentation}, there is a computable partition such that 
neither half is even classically isomorphic to $G$. Therefore, to study Theorem \ref{thm:Cameron} in $\mathsf{REC}$, we can restrict 
our attention to computable graphs 
that have isolated or universal vertices. Moreover, since isolated nodes in $G$ correspond to universal nodes in $\overline{G}$, we can apply 
Proposition \ref{prop:duality} to restrict to computable graphs that have isolated nodes. It follows that Theorem \ref{thm:Cameron} holds in $\mathsf{REC}$ if and only 
if for every computable graph $G$ that has isolated nodes but is not isomorphic to $\overline{K}_{\omega}$, there is a computable partition 
$G = X_0 \sqcup X_1$ such that neither $X_0$ nor $X_1$ is computably isomorphic to $G$. 
We establish this statement under the additional hypothesis that the set of vertices of finite degree is computably enumerable. 

\begin{thm}
\label{thm:REC}
Let $G$ be a computable graph that has isolated vertices but is not isomorphic to $\overline{K}_{\omega}$. If the set of vertices with finite 
degree is c.e., then there is a computable partition $G = X_0 \sqcup X_1$ such that neither $X_0$ nor $X_1$ is computably isomorphic to $G$. 
\end{thm}

The corollary follows from Theorem \ref{thm:REC} and Proposition \ref{prop:duality}.

\begin{cor}
Let $G$ be a computable graph that has universal vertices but is not isomorphic to $K_{\omega}$. If the set of vertices with cofinite 
degree is c.e., then there is a computable partition $G = X_0 \sqcup X_1$ such that neither $X_0$ nor $X_1$ is computably isomorphic to $G$.
\end{cor}

We now give the proof of Theorem \ref{thm:REC}.

\begin{proof}
Without loss of generality, assume the set of vertices of $G$ is $\omega$. 
If the set of isolated nodes is computable, then let $X_0$ be the set of isolated nodes and $X_1 = G \setminus X_0$ as in the proof of 
Theorem \ref{thm:Cameron}. In this case, neither $X_0$ nor $X_1$ is even classically isomorphic to $G$. 
Therefore, assume the set of isolated nodes is not computable, and so in particular, is infinite. 

The construction proceeds in stages with $G_s$ denoting the subgraph of $G$ on $\{ 0, \ldots, s\}$. At stage $s$, we determine whether to put $s$ in $X_0$ 
or $X_1$. Following the usual use conventions, if $\Phi_{e,s}(x) = y$, then $x,y < s$, so $x \in G_s$ and $y$ has already been placed in either 
$X_0$ or $X_1$. 

For each $e \in \omega$, we need the partition to satisfy the following requirement.  
\[
R_e: \Phi_e \text{ is not an isomorphism from } G \text{ to } X_0 \text{ or } X_1. 
\]
The strategy for $R_e$ keeps five parameters: numbers $i_e$ and $x_e$, finite sets $D_e$ and $S_e$, and a binary string $\sigma_e$. 
The parameter $i_e$ is defined when $\Phi_e(0)$ converges and is set such that $\Phi_e(0) \in X_{i_e}$, indicating we must work to prevent $\Phi_e$ from 
being an isomorphism onto $X_{i_e}$. Unlike the other parameters, $i_e$ does not change values once it is defined.

The goal is to make $\Phi_e$ map an isolated node in $G$ to a non-isolated node in $X_{i_e}$ or vice versa. 
The parameter $x_e$ marks the isolated node we are currently working with. As the value of $x_e$ grows, we attempt to compute the set of isolated nodes in 
$G$, defining (and later extending) $\sigma_e$ to be an initial segment of this computable function. Eventually, because the set of isolated vertices is not 
computable, $\sigma_e$ must be wrong about some vertex, and this incorrect vertex will be a    
diagonalizing value for $\Phi_e$. The sets $D_e$ and $S_e$ contain nodes related to commitments $R_e$ makes about putting future vertices into 
$X_{i_e}$ or $X_{1-i_e}$ as we try to make specific nodes in $X_{i_e}$ isolated or not. 

The construction for a single $R_e$ works as follows. Suppose at stage $s_0$ we define $i_e$ such that $\Phi_e(0) \in X_{i_e}$. We set $x_e$ to be the least 
vertex that currently looks isolated in $G$. Since $G_{s_0}$ might not contain any isolated vertices, we may need to look at vertices in $G_t$ for $t \geq s_0$ to find 
a vertex that is isolated in $G_t$. At this point, we know the vertices $v < x_e$ are not isolated in $G$, but we are unsure whether $x_e$ will be 
isolated or not. We record this information by defining $\sigma_e$ with $|\sigma_e| = x_e$ and $\sigma_e(v) = 0$ for $v <x_e$. 

We do nothing more until $\Phi_e(x_e) $ converges. Assume $\Phi_e(x_e) = y  \in X_{i_e}$ else we win $R_e$ trivially. 
Our goal is to use $y$ to guess whether $x_e$ will be isolated in $G$ in such a way that if our guess is wrong, $R_e$ will be met. Since we are defining 
$X_0$ and $X_1$, we have some control over whether $y$ will be isolated in $X_{i_e}$. 

We split into three substrategies. First, if $y$ already has a neighbor in $X_{i_e}$, then we know $y$ is not isolated in $X_{i_e}$. We declare that 
$x_e$ will not be isolated in $G$ by setting $\sigma_e(x_e) = 0$. If $\sigma_e$ turns out to be wrong about $x_e$, then we win 
$R_e$ because $\Phi_e(x_e) = y$ with $x_e$ isolated in $G$ and $y$ not isolated in $X_{i_e}$. 

Assume $y$ does not currently have a neighbor in $X_{i_e}$. To determine whether to follow the second or third substrategy, we use the hypothesis that the nodes of finite 
degree form a c.e.~set. In parallel, we enumerate the vertices of finite degree searching for $y$, and we look ahead in $G$ to see if $y$ gains a future neighbor which is not yet 
promised to be put in $X_0$ or $X_1$. At least one of these searches must succeed as $R_e$ (and later, even higher priority requirements) will only have 
made finitely many future commitments. The search that terminates first determines which substrategy we follow. 

If we see $y$ enumerated in the set of vertices with finite degree, we promise to put all of $y$'s future neighbors into $X_{1-i_e}$ to make $y$ isolated in $X_{i_e}$. 
To keep track of this commitment, we place $y$ in $D_e$. We declare $x_e$ will be 
isolated by setting $\sigma_e(x_e) = 1$. As long as we keep our promise, if $\sigma_e$ turns out to be wrong about $x_e$, then 
we win $R_e$ because $\Phi_e$ maps a non-isolated node $x_e$ in $G$ to an isolated node $y$ in $X_{i_e}$. 

If we find a future uncommitted neighbor $v$ of $y$, we promise to put $v$ into $X_{i_e}$ at stage $v$ and we mark this commitment by 
putting $v$ into $S_e$. We declare $x_e$ will not be isolated by setting $\sigma_e(x_e) = 0$. As long as we keep our 
promise to put $v$ in $X_{i_e}$, $y$ will not be isolated in $X_{i_e}$. Therefore, again, we win $R_e$ if $\sigma_e$ is incorrect about $x_e$. 

Once we have defined $\sigma_e(x_e)$, we repeat the process above. We define $x_{e,s}$ to be the next largest number that currently 
looks isolated, set $\sigma_e(v) = 0$ for $x_{e,s-1} < v < x_{e,s}$, wait for $\Phi_e(x_e)$ to converge, and employ the appropriate substrategy to define $\sigma_e$ on the new 
value of $x_e$. This process cannot repeat infinitely 
because the set of isolated nodes is not computable. Therefore, we must eventually see a true diagonalization that satisfies $R_e$.

The strategies for different $R$ requirements interact in a standard finite injury way with the priority determined by the index on $R_e$. 
If more than one strategy has an opinion about whether to place the vertex $s$ into $X_0$ or $X_1$ 
at stage $s$, we follow the higher priority strategy and initialize the lower priority one. 

One feature to note is that when $x_e$ is defined at stage $s$, it is currently isolated in $G_s$ (or possibly in $G_t$ for some $t >s$). 
By the time $\Phi_e(x_e)$ converges, $x_e$ may not longer be isolated. However, that doesn't make any difference for our 
strategies. We use the definitions of $X_0$ and $X_1$ to force graph theoretic behavior on the image side with no regard to whether the vertex $x_e$ has remained 
isolated after the stage at which the parameter is assigned. 

We give the formal construction. A vertex $v$ is \textit{claimed by} $R_e$ if $v \in S_e$ or $v$ is connected to a vertex in $D_e$. When $R_e$ is 
\textit{initialized}, its parameters $i_e$, $x_e$ and $\sigma_e$ are undefined and the sets $D_e$ and $S_e$ are set to $\emptyset$. 
$R_e$ \textit{looks satisfied at stage} $s$ if any of the following are true.
\begin{itemize}
\item[(S1)] $\Phi_{e,s}(0)$ diverges. 
\item[(S2)] $(\exists a \neq b < s) \, \big(\Phi_{e,s}(a) \! \! \downarrow = \Phi_{e,s}(b) \! \! \downarrow \vee (\Phi_{e,s}(a) \in X_0 \wedge \Phi_{e,s}(b) \in X_1) \big)$. 
\item[(S3)] $x_e$ is defined but $\Phi_{e,s}(x_e)$ diverges. 
\item[(S4)] $(\exists x <s) (\exists y \in D_e) \, \big( \Phi_{e,s}(x) = y \wedge x \text{ is not isolated} \big)$. 
\item[(S5)] $(\exists x < s) (\exists y,z) \, \big( x \text{ is isolated } \wedge \Phi_{e,s}(x)=y\wedge E(y,z) \wedge z \in S_e \big)$. 
\end{itemize}

At stage $0$, initialize all requirements and put $0$ into $X_0$. 
At stage $s>0$, let each $R_e$ with $e < s$ act in order as described in the $R_e$ module below. When these requirements are done acting, 
check if there is an $e < s$ such that $s$ is claimed by $R_e$. If not, put $s$ into $X_0$. If so, let $e$ be the least such index. If 
$s \in S_e$, put $s$ into $X_{i_e}$ and otherwise put $s$ into $X_{1-i_e}$. Initialize all $R_i$ with $i>e$ and end the stage. 
For the $R_e$ module, act in the first case below that applies.

\textit{Case 1. $R_e$ looks satisfied at $s$.} Do nothing and go to the next requirement. 

\textit{Case 2. $i_e$ is not defined.} Since $R_e$ does not look satisfied, $\Phi_{e,s}(0)$ must converge. By use conventions, $\Phi_{e,s}(0) < s$, so $\Phi_e(0)$ is already  
in $X_0$ or $X_1$. Set $i_e$ such that $\Phi_e(0) \in X_{i_e}$ and go to the next requirement. 

\textit{Case 3. $x_e$ is not defined.} In this case, $\sigma_e$ is also undefined. Set $x_{e,s}$ to be the least vertex that is isolated in some $G_t$ for $t \geq s$. 
Define $\sigma_e$ with $|\sigma_e| = x_e$ and $\sigma_e(v) = 0$ for all $v < x_e$. Go to the next requirement. 

\textit{Case 4. $x_e$ is defined.} Since $x_e$ is defined and $R_e$ does not look satisfied, $\Phi_{e,s}(x_e)$ must converge. Let $y_e = \Phi_e(x_e)$ and note that 
$y_e \in X_{i_e}$. 
\begin{itemize}
\item[(4.1)] If $y_e$ has a neighbor in $X_{i_e}$, define $x_{e,s}$ and $\sigma_{e,s}$ as described after (4.2). 
\item[(4.2)] Otherwise, dovetail the enumerations of the finite degree vertices and of the neighbors of $y_e$ until one of (4.2.1) or (4.2.2) occurs.   
\begin{enumerate}
\item[(4.2.1)] $y_e$ is enumerated as a vertex with finite degree. 
\item[(4.2.2)] $y_e$ gets a new neighbor that is not protected by a requirement $R_i$ with $i < e$. 
\end{enumerate}
If (4.2.1) halts first, put $y_e$ into $D_e$. If (4.2.2) halts first, put the neighbor into $S_e$. In either case, define $x_{e,s}$ and $\sigma_{e,s}$ as described below. 
\end{itemize}
Set $x_{e,s}$ to be the least vertex $v > x_{e,s-1}$ that is isolated in $G_t$ for some $t \geq s$. Define $\sigma_{e,s}$ to be an extension of $\sigma_{e,s-1}$ of length  
$x_{e,s}$. If we acted in (4.2.1), set $\sigma_{e,s}(x_{e,s-1}) = 1$, and if we acted in (4.1) or (4.2.2), set $\sigma_{e,s}(x_{e,s-1}) = 0$. In either case, 
set $\sigma_{e,s}(v) = 0$ for $x_{e,s-1} < v < x_{e,s}$. 

This completes the construction. We note some properties that are clear by inspection and that we use implicitly below. 
First, for any index $e$ and stage $s$, $x_e$ is defined if and only if $\sigma_e$ is defined. 
Second, if $y \in D_e$, then $y$ has finite degree in $G$. Moreover, if $y$ is placed in $D_e$ at stage $s$, then $y$ is isolated in $G_s$. Third, at any stage $s$, each 
requirement $R_i$ has claimed only finitely many vertices. Therefore, in (4.2), if $y_e$ has infinite degree, $R_e$ will eventually see a vertex connected to $y_e$ that is not 
claimed by any $R_i$ with $i < e$. 

\begin{lem}
\label{lem:notiso}
Consider an index $e$ and a stage $s$ such that $\sigma_{e,s}$ is defined. Let $t < s$ be the last stage at which $R_e$ was initialized. For a vertex 
$v < |\sigma_{e,s}|$, let $t_v > t$ be the least stage such that $|\sigma_{e,t_v}| > v$. 
If $v \neq x_{e,t_v-1}$, then $v$ is not isolated in $G$ and $\sigma_{e,t_v}(v) = 0$.  
\end{lem} 

\begin{proof}
By the hypotheses, at stage $t_v-1$, either $\sigma_e$ is undefined or $|\sigma_{e,t_v-1}| \leq v$. In the former case, $R_e$ acts in Case 3 at $t_v$ to define 
$\sigma_{e,t_v}$, and in the latter case, $R_e$ acts in Case 4 to extend $\sigma_{e,t_v-1}$ to $\sigma_{e,t_v}$. The arguments in each case are essentially the same, so we 
assume that $R_e$ acts in Case 4. 

The parameter $x_{e,t_v}$ is defined to be the least vertex greater than $x_{e,t_v-1}$ that is isolated in $G_t$ for some $t \geq t_v$. Fix the stage $t \geq t_v$ such that 
$x_{e,t_v}$ is isolated in $G_t$. Since $v < x_{e,t_v}$ was not chosen as the value of the parameter, it follows that $v$ is not isolated in $G_t$ and hence is not isolated in 
$G$. Furthermore, since $x_{e,t_v-1} < v< x_{e,t_v}$, we set $\sigma_{e,t_v}(v) = 0$.  
\end{proof}

\begin{lem}
\label{lem:P12}
For each $e$, the following properties hold. 
\begin{enumerate}
\item[(P1)] $R_e$ is initialized only finitely often. 
\item[(P2)] There is a stage $t$ such that for all $s \geq t$, $R_e$ looks satisfied at $s$. 
\end{enumerate}
\end{lem}

\begin{proof}
We prove the properties simultaneously by induction on $e$. First, consider (P1). This property holds trivially for $R_0$. For $e>0$, 
fix a stage $t$ such that for all $j < e$, $R_j$ is never initialized after $t$ and $R_j$ looks satisfied at $s$ for all $s \geq t$. By construction, the parameters for 
$R_j$ do not change after $t$. Since each $y \in D_j$ has finite degree, 
the sets $D_j$ and $S_j$ can cause a vertex $s$ to be put into $X_{i_j}$ or $X_{1-i_j}$ only finitely often after stage $t$. In particular, each $R_j$ can only initialize 
$R_e$ finitely often and so (P1) holds for $e$. 

For (P2), fix $e$. Let $t$ be the last stage at which $R_e$ is initialized. Suppose for a contradiction, there is no stage $t$ as in (P2). By (S1)-(S3), we must have that  
$\Phi_e(0)$ converges, $\Phi_e$ is 1-to-1 and maps into $X_{i_e}$, and $\Phi_e(x_e)$ converges for each value of the parameter $x_e$ after stage $t$. 
It follows that, after stage $t$, $x_e$ takes on an infinite sequence of values $z_0 < z_1 < \ldots$. 

Let $s_k$ be the stage at which $R_e$ acts in Case 4 to set $x_{e,s_k} = z_k$ and define $\sigma_{e,s_k}$ with length $x_{e,s_k}$. By construction, the sequences 
$\sigma_{e,s_k}$ are nested and uniformly computable, so $g_e = \cup_k \sigma_{e,s_k}$ is a computable function. To finish the proof, it suffices to show that $g_e$ is the 
characteristic function for the set of isolated nodes as this provides the desired contradiction. 

By Lemma \ref{lem:notiso}, if $v \neq z_k$ for all $k$, then $v$ is not isolated in $G$ and $g_e(v) = \sigma_{e,s_\ell}(v) = 0$, where $\ell$ is least such that 
$v < z_\ell$. Therefore, $g_e$ is correct on all nodes not of the form $z_k$. 

The value of $g_e(z_k)$ is set at stage $s_{k+1}$ when $R_e$ sees $\Phi_{e,s_{k+1}}(z_k)$ converge and defines $\sigma_{e,s_{k+1}}(z_k)$ in 
Case 4. Let $y_k = \Phi_e(z_k)$ and note that by (S2), $y_e \in X_{i_e}$. 
At stage $s_{k+1}$, $R_e$ either acts in (4.1), puts a neighbor $v_k$ of $y_k$ into $S_e$ via (4.2.2), or puts $y_k$ into $D_e$ via (4.2.1).

Suppose $R_e$ acts in (4.1) or puts $v_k$ in $S_e$ in (4.2.2). By construction, $g_e(z_k) = \sigma_{e,s_{k+1}}(z_k) = 0$, so we need to show $z_k$ is not 
isolated in $G$. We claim that $y_k$ will not be isolated in $X_{i_e}$. If $R_e$ acts in (4.1), then  $y_k$ is already not isolated in $X_{i_e, s_{k+1}}$. 
Otherwise, $R_e$ acts in (4.2.2), and because $R_e$ is not initialized after stage $t$, the vertex $v_k$ remains in $S_e$ until it is placed in 
$X_{i_e}$ at stage $v_k$ making $y_k$ not isolated in $X_{i_e}$. It now follows by 
(S5) that $z_k$ must eventually get a neighbor in $G$.

Finally, suppose $R_e$ puts $y_k$ into $D_e$ via (4.2.1). By construction, $g_e(z_k) = \sigma_{e,s_{k+1}}(z_k) = 1$, so we need to show  
$z_k$ is isolated in $G$. Since $R_e$ is not initialized again, $y_k$ remains in $D_e$ at all future stages. If $z_k$ were to get a neighbor at a future stage $s'$, then 
(S4) would be true at every $s \geq s'$, contradicting the assumption that there is no stage $t$ as in (P2). 
\end{proof}

To complete the proof of Theorem \ref{thm:REC}, we show each $R_e$ is satisfied. Fix $e$, let $t_0$ be the last stage at which $R_e$ is initialized and let $t_1 \geq t_0$ be the stage 
from (P2). 
If $R_e$ looks satisfied for all $s \geq t_1$ because of (S1), (S2) or (S3), then $R_e$ is satisfied because $\Phi_e$ is either not total, not one-to-one, or doesn't map into a 
single $X_i$. 

Suppose $R_e$ looks satisfied for all $s \geq t_1$ because of (S4) with $\Phi_{e,t_1}(x) = y \in D_e$. The vertex $y$ must have been placed in $D_e$ after stage $t_0$, so  
$R_e$ is not initialized after $y$ enters $D_e$. It follows that no higher priority requirement overrides the $R_e$ commitment to put $y$'s neighbors into $X_{1-i_e}$. Therefore, 
eventually all of $y$'s (finitely many) neighbors are in $X_{1-i_e}$, so $y$ is, in fact, isolated in $X_{i_e}$. $R_e$ is won because $x$ is not isolated in $G$ but $y$ is 
isolated in $X_{i_e}$. 

Finally, suppose $R_e$ looks satisfied because of (S5) with $\Phi_{e,t_1}(x) = y$ and an edge $E(y,z)$ with $z \in S_e$. As in the previous paragraph, 
$R_e$ keeps its commitment to put $z$ into $X_{i_e}$. Therefore, $y$ is not isolated in $X_{i_e}$, so $R_e$ is won because $x$ is isolated in $G$ and $y$ is not isolated in $X_{i_e}$.
\end{proof}

\section{Induction aspects}
\label{sec:induction}

In the classical proof of Theorem \ref{thm:Cameron}, we used the least number principle $\mathsf{L}\Sigma^0_2$ to conclude that if $G$ is not isomorphic to 
$\mathcal{R}$, then there is a least $n$ such that there exist finite sets $A$ and $B$ with $|A|+|B|=n$ for which the extension axiom $\varphi_{\mathcal{R}}$ 
fails. In this section, we show that the statement asserting the existence of a least such $n$ is equivalent to $\mathsf{L}\Sigma^0_2$. 

Let $G$ be a graph. $\langle X_0, X_1 \rangle$ is an $n$-\textit{extension pair} if $X_0$ and $X_1$ are 
disjoint subsets of $G$ with $|X_0|+|X_1| = n$. The \textit{extension property} holds for $\langle X_0, X_1 \rangle$ in $G$ if there is 
a vertex $v \in G \setminus (X_0 \cup X_1)$ that is connected to every vertex in $X_0$ and to none of the vertices in $X_1$. By definition, $G$ is a random graph if, for all $n$, every 
$n$-extension pair in $G$ has the extension property. 

Keep in mind two properties of extension pairs during the construction. First, 
the only $0$-extension pair is $\langle \emptyset, \emptyset \rangle$, which always has the extension property since $G$ is nonempty. Therefore, if $G \not \cong \mathcal{R}$, 
then the least $n$ for which the extension property fails satisfies $n \geq 1$. 
Second, if $G$ is infinite and has an $m$-extension pair $\langle X_0,X_1 \rangle$ without the extension property, then for every $n \geq m$, we can form  
an $n$-extension pair without  the extension property by adding $n-m$ many vertices from $G \setminus (X_0 \cup X_1)$ to $X_0$.

\begin{thm}
\label{thm:ind}
The following are equivalent over $\mathsf{RCA}_0$.
\begin{enumerate}
\item[(1)] For every graph $G$ not isomorphic to $\mathcal{R}$, there is a least $n$ for which there is an $n$-extension pair $\langle X_0, X_1 \rangle$ 
that fails to have the extension property.
\item[(2)] $\mathsf{L}\Sigma^0_2$.
\end{enumerate}
\end{thm}

\begin{proof}
(2) implies (1) because the formula with free variable $n$ saying ``there is an $n$-extension pair that fails to have the extension property'' is $\Sigma^0_2$. 

For (1) implies (2), fix a $\Sigma^0_2$ formula $\psi(n)$ of the form $\exists x \forall y \phi(n,x,y)$, where $\phi$ is $\Sigma^0_0$, such that $\psi(n)$ holds for some $n$. 
Without loss of generality, we can assume $\neg \psi(0)$ by replacing $\psi(n)$ by $n>0 \wedge \psi(n-1)$ if necessary. In addition, 
it suffices to construct $G$ such that $(\exists m \leq n) \psi(m)$ holds if and only if for some $m \leq n$, 
there is an $m$-extension pair without the extension property.  

We construct $G$ in stages with $G_s$ denoting the graph at stage $s$. 
For each $n \geq 1$, we have a strategy which tries to ensure that $\psi(n)$ holds if and only if the 
extension property fails for a pair $\langle F, \emptyset \rangle$ with $|F|=n$. 
The strategy for $n$ keeps two parameters: $x_n$ and $F_n$. The parameter $x_n$ is the existential witness we are checking in the formula 
$\psi(n)$ and $F_n$ 
is a set of size $n$. As long as $(\forall y \leq s) \phi(n,x_n,y)$ holds, we prevent any node in $G_s$ from connecting to every vertex in $F_n$. However, if 
$(\exists y \leq s) \neg \phi(n,x_n,y)$, then we add new nodes to witness the extension property for every $n$-extension pair in $G_s$ including $\langle F_n, \emptyset \rangle$.   
We increment $x_n$ and add $n$ new elements to form a new set $F_n$. 

This strategy succeeds in isolation. If $\psi(n)$ holds, then $x_n$ eventually reaches a value for which $\forall y \phi(n,x_n,y)$. We choose a final set 
$F_n$ and prohibit any node from connecting to all of $F_n$, making the extension property fail for $\langle F_n, \emptyset \rangle$. 
On the other hand, if $\neg \psi(n)$ holds, then for every value of $x_n$, there is a stage $s$ such that 
$(\exists y \leq s) \neg \phi(n,x_n,y)$, at which point we increase $x_n$ and add witnesses for all $n$-extension pairs in $G_s$. 
Since each $n$-extension pair $\langle X_0, X_1 \rangle$ in $G$ is contained in $G_s$ for large enough $s$, the extension property holds for every $n$-extension pair.

Unfortunately, these strategies interfere with each other. Consider $n_0 < n_1$. When $n_1$ adds witnesses to realize the extension 
property for every $n_1$-extension pair, it adds nodes connected to every point in $F_{n_0}$. To protect $n_0$, we restrict $n_1$ from adding a witness 
for $\langle X_0, X_1 \rangle$ when $F_{n_0} \subseteq X_0$. When $\neg \psi(n_0)$ holds, this restriction has no effect in 
the limit because the parameter 
$F_{n_0}$ never settles on a final set. However, when $\psi(n_0)$ holds, it prevents $n_1$ from realizing the extension property for sets extending the final value of 
$\langle F_{n_0}, \emptyset \rangle$. It also guarantees there will be $n_1$-extension pairs without the extension property, a condition that is necessary 
considering the comments before Theorem \ref{thm:ind}. 

We give the full construction of $G$. For each $n \geq 1$, we keep parameters $x_{n,s} \in \mathbb{N}$ and $F_{n,s}$ with $|F_{n,s}|=n$ defined 
by primitive recursion on $s$. We set default values $x_{n,s} = 0$ and $F_{n,s} = \{ 0, \ldots, n-1 \}$ for $n > s$. 

We define $G_s$ and $m_s$ by primitive recursion on $s$. $G_s = \langle V_s, E_s \rangle$ with vertex set $V_s = \{ 0, \ldots, m_s \}$ and edge  
relation $E_s \subseteq V_s \times V_s$. The values of $m_s$ are strictly increasing and unbounded, so the vertex set of 
$G$ is $\cup_s V_s = \mathbb{N}$. We maintain $E_{s+1} \! \upharpoonright \! V_s = E_s$ so $E_G = \cup_s E_s$ is definable with bounded quantifiers:
\[
E_G(n,m) \text{ holds if and only if } E_{\max \{ n,m \}}(n,m) \text{ holds.}
\]

For $s=0$, set $m_0 = 0$, $V_0 = \{ 0 \}$ and $E_0 = \emptyset$ with the default parameters $x_{1,0} = 0$ and $F_{1,0} = \{ 0 \}$. (This is a throwaway stage because the 
strategies are only for $n \geq 1$.) 

For $s=1$, we set the initial real parameters for the $n=1$ strategy. Set $x_{1,1} = 0$, $m_1 = 1$ (so $V_1 = \{ 0, 1 \}$), $E_1 = \emptyset$ and $F_{1,1} = \{ 1 \}$. 

For $s > 1$, we determine which $n < s$ need to act. Define 
\[
Y_s = \{ n : 1 \leq n < s \text{ and } (\exists y \leq s) \neg \phi(n,x_{n,s-1},y) \}.
\]
If $Y_s = \emptyset$, then we do not act for any $1 \leq n < s$. Set $m_s = m_{s-1}+s$ to add $s$ new elements to $G_s$, but keep $E_s = E_{s-1}$ so there are no  
edges between the new vertices and the nodes in $G_s$. Set $x_{s,s} = 0$ and let $F_{s,s} = \{ m_{s-1}+1, \ldots, m_s \}$ be the set of $s$ many new elements. 
For $1 \leq n < s$, leave the parameters unchanged: $x_{n,s} = x_{n,s-1}$ and $F_{n,s} = F_{n,s-1}$. 

If $Y_s \neq \emptyset$, then we act for each $n \in Y_s$. For $n \in Y_s$, let $k_n$ be the number of $n$-extension pairs $\langle X_0, X_1 \rangle$  in 
$G_{s-1}$ such that there is no $m < n$ with $m \not \in Y_s$ and $F_{m,s-1} \subseteq X_0$. We refer to such a pair as an \textit{active} $n$-extension pair. 

Set $m_s = m_{s-1} + s + \sum_{n \in Y_s} (k_n + n)$. Use the first $\sum_{n \in Y_s} k_n$ new elements as follows. Order the active $n$-extension pairs 
$\langle X_0, X_1 \rangle$ for $n \in Y_s$ and, considering these pairs in order, put edges from the next new element in $G_s$ to the nodes in $X_0$ 
(and to no other nodes). These are the only new edges added to $G_s$. 

If $n \not \in Y_s$, keep $x_{n,s} = x_{n,s-1}$ and $F_{n,s} = F_{n,s-1}$. If $n \in Y_s$, set $x_{n,s}$ to be the 
least $x \leq s$ such that $(\forall y \leq s) \phi(n,x,y)$ holds. If there is no such $x \leq s$, then set $x_{n,s} = s+1$. Use the next $\sum_{n \in Y_s} n$ elements 
to define $F_{n,s}$ for $n \in Y_s$. Consider these $n$ in order, setting $F_{n,s}$ to be the set of the next $n$ many new elements in 
$G_s$. Finally, set $x_{s,s} = 0$ and $F_{s,s}$ to be the set of the remaining $s$ many new elements in $G_s$. This completes the construction at stage $s$.

By $\Sigma^0_0$ induction on $s$, it follows that for all $s$ and $1 \leq n \leq s$, $F_{n,s}$ is a set of size $n$ with $x_{n,s} < \min F_{n,s}$ and there is no $z \in G_{s}$ 
connected to all the nodes in $F_{n,s}$. In addition, $x_{n,s} \leq x_{n,s+1}$, and, by $\Sigma^0_1$ induction on $s$, with $x$ as a parameter, we have that the condition 
$\forall y \phi(n,x,y)$ implies that $x_{n,s} \leq x$.

Suppose $\psi(n)$ holds. By $\mathsf{L}\Pi^0_1$ (which holds in $\mathsf{RCA}_0$), there is a least $x$ such that $\forall y \phi(n,x,y)$. Since $x$ is chosen least, 
$(\forall u<x) \exists y \neg \phi(n,u,y)$ holds, and by $\mathsf{B}\Sigma^0_0$, there is a $t$ such that 
$(\forall u<x) (\exists y < t) \neg \phi(n,u,y)$. Without loss of generality, $t > n$. We claim that $x_{n,t} = x$. If $x_{n,t-1} = x$, then we are done since 
$x_{n,t-1} \leq x_{n,t} \leq x$ by the previous paragraph. If $x_{n,t-1} < x$, then  
$n \in Y_{t}$ and we set $x_{n,t} = x$ because $x$ is the least such that  $(\forall y \leq t) \phi(n,x,y)$. 

It follows from the claim that $x_{n,s} = x_{n,t} = x$ for every $s \geq t$. Therefore, $\lim_s x_{n,s} = x$. 
Moreover, setting $F_n = F_{n,t}$, we have $F_{n,s} = F_n$ for all $s \geq t$ and so $\lim_s F_{n,s} = F_n$. 

Similarly, if $\neg \psi(n)$ holds, then the values of $x_{n,s}$ are unbounded as $s$ goes to infinity, as are the values 
of the minimum elements of $F_{n,s}$. 

To complete the proof, we show that for all $n$, $(\exists m \leq n) \psi(m)$ holds if and only if there is an $m \leq n$ and an $m$-extension pair 
$\langle X_0, X_1 \rangle$ that fails to have the extension property in $G$. 

For the forward direction, fix $n \geq 1$ such that $(\exists m \leq n) \psi(m)$ holds and fix $m \leq n$ with $\psi(m)$. By $\mathsf{L}\Pi^0_1$,  
fix the least $t$ such that $F_{m,s} = F_{m,t}$ for all $s \geq t$ and let $F_m = F_{m,t}$. We show the extension property fails in $G$ for $\langle F_m, \emptyset 
\rangle$. It suffices to show by $\Sigma^0_0$ induction that for all $s \geq t$, there is no node in $G_s$ connected to all the nodes in $F_m$. 

For $s=t$, since $t$ is chosen least, the set $F_{m,t} \neq F_{m,t-1}$ consists of $m$ many new elements added to $G_t$ none of which are connected to nodes in 
$G_t$. Therefore, the condition holds for $s=t$. 

Assume the condition holds for a fixed $s \geq t$. A new node $v \in G_{s+1}$ is only connected to a node in $G_s$ when $v$ is 
used to witness the extension property for a $k$-extension pair $\langle X_0, X_1 \rangle$ with $k \in Y_{s+1}$. In this case, $v$ is only  
connected to the nodes in $X_0$. If $k < m$, then $|X_0| \leq k$, so $v$ cannot be connected to every node in $F_m$. 
If $k > m$, then by construction, we have $F_{m,s} \not \subseteq X_0$ and hence $v$ is not connected to every node in $F_m$. 

For the backward direction, assume $(\forall m \leq n) \neg \psi(m)$. Fix $m \leq n$ and an $m$-extension pair $\langle X_0, X_1 \rangle$ in $G$. We need to
show this pair has the extension property. Fix $t_0$ such that $X_0, X_1 \subseteq G_{t_0}$. 
Suppose $X_0 = \emptyset$. In this case, we need to show there is a node $v \in G$ which is not connected 
to any node in $X_1$. However, at each stage $s > 0$, we add new elements to $G_s$ that are not connected to any node in $G_{s-1}$. 

Suppose $X_0 \neq \emptyset$ and let $\ell$ be the least element of $X_0$. Since $(\forall k < m) \neg \psi(k)$, 
we have $(\forall k < m) (\forall x \leq \ell) \exists y \neg \phi(k,x,y)$. By $\mathsf{B}\Sigma^0_0$, we can fix $t > t_0$ such that 
$(\forall k < m) (\forall x \leq \ell) (\exists y < t) \neg \phi(k,x,y)$. It follows that $x_{k,t} > \ell$ for all $k < m$, and hence $\min F_{k,t} > \ell$ for all 
$k < m$. In particular, for all $s \geq t$ and $k < m$, $F_{k,s} \not \subseteq X_0$. Therefore, for any $s > t$ at which we act for $m$, we add 
an extension witness for the pair $\langle X_0, X_1 \rangle$ as required.

\end{proof}

\end{document}